\documentclass{birkjour}
\usepackage{amsfonts}
\usepackage{amssymb,amsthm}
\usepackage{amsmath,comment}
\def\appendixname{{\bf Appendix}}

\newcommand{\R}{\mathbb{R}}
\newcommand{\N}{\mathbb{N}}

\newcommand{\Z}{\mathbb{Z}}

\newtheorem{theorem}{Theorem}[section]
\newtheorem{lemma}[theorem]{Lemma}

 \theoremstyle{definition}
\newtheorem{definition}[theorem]{Definition}

 \theoremstyle{remark}
 \numberwithin{equation}{section}

\begin{document}
\title{Heat content in non-compact Riemannian manifolds}
\author{M. van den Berg}
%EndAName
\address{School of Mathematics\\
University of Bristol\\
University Walk, Bristol BS8 1TW\\
United Kingdom}
\email{{mamvdb@bristol.ac.uk}}


\date{31 December 2018}
\begin{abstract}\noindent Let $\Omega$ be an open set
in a complete, smooth, non-compact, $m$-dimensional Riemannian manifold $M$ without boundary, where $M$ satisfies a two-sided Li-Yau gaussian heat kernel bound. It is shown that if $\Omega$ has infinite measure, and if $\Omega$ has finite heat content $H_{\Omega}(T)$ for some $T>0$, then $H_{\Omega}(t)<\infty$ for all $t>0$. Comparable two-sided bounds for $H_{\Omega}(t)$ are obtained for such $\Omega$.
\end{abstract}
\maketitle

%\mbox{}\newpage

\section{Introduction\label{sec1}}
Let $(M,g)$ be a complete, smooth, non-compact, $m$-dimensional Riemannian
manifold without boundary, and let $\Delta$ be the
Laplace-Beltrami operator acting in $L^2(M)$. It is well known (see
\cite{EBD3}, \cite{GB}, \cite{LSC}, \cite{RS}) that the heat equation
\begin{equation}\label{e1}
\Delta u(x;t)=\frac{\partial u(x;t)}{\partial t},\quad x\in M,\quad t>0, \
\end{equation}
has a unique, minimal, positive fundamental solution $p_M(x,y;t)$
where $ x\in M$, $y \in M$, $t>0$. This solution, the
heat kernel for $M$, is symmetric in $x,y$, strictly positive,
jointly smooth in $x,y\in M$ and $t>0$, and it satisfies the
semigroup property
\begin{equation} \label{e2}
p_M(x,y;s+t)=\int_{M}dz\ p_M(x,z;s)p_M(z,y;t),
\end{equation}
for all $x,y\in M$ and all $t,s>0$, where $dz $ is the Riemannian
measure on $M$.

We define the heat content of an open set $\Omega$ in $M$ with boundary $\partial \Omega$ at $t$ by
\begin{equation*}%\label{e3}
    H_{\Omega}(t)=\int_{\Omega}dx\,\int_{\Omega} dy\, p_M(x,y;t).
\end{equation*}

It was shown (\cite{vdBG2}) that if $\Omega$ is non-empty, bounded, and $\partial\Omega$ is of class
$C^{\infty}$, and if $(M,g)$ satisfies exactly one of the following three conditions: (i) $M$ is compact and without
boundary, (ii) $(M,g)=(\mathbb{R}^m,g_e)$ where $g_e$ is the usual Euclidean metric on $\mathbb{R}^m$, (iii) $M$ is a compact
submanifold of $\mathbb{R}^m$ with smooth boundary and $g=g_e|_M$, then there exists a complete asymptotic series such that
\begin{equation}\label{e4}
H_{\Omega}(t)=\sum_{j=0}^{J-1}\beta_jt^{j/2}+O(t^{J/2}),\
t\downarrow 0,
\end{equation}
where $J\in \N$ is arbitrary, and where the $\beta_j:j=0,1,2,\dots$ are locally computable geometric invariants. In particular,
we have that
\begin{equation*}%\label{e5}
\beta_0=\vert\Omega\vert,\,  \beta_1=-\pi^{-1/2}\textup{Per}(\Omega),\,  \beta_2=0,
\end{equation*}
where $\vert \Omega\vert$ is the measure of $\Omega$, and $\textup{Per}(\Omega)$ is the perimeter of $\Omega$.

For earlier results in the Euclidean setting we refer to
\cite{MPPP1}, \cite{MPPP2}, \cite {P}, and subsequently to \cite{mvdB13}, \cite{vdBGi}, and \cite{vdBGi2}.

Define $u_{\Omega}:\Omega \times (0,\infty)\mapsto \R$ by
\begin{equation}\label{e6}
    u_{\Omega}(x;t)=\int_{\Omega}dy\, p_M(x,y;t).
\end{equation}
Then $u_{\Omega}$ is a solution of the heat equation \eqref{e1} and satisfies
\begin{equation}\label{e7}
  \lim_{t\downarrow 0}  u_{\Omega}(x;t)=\textbf{1}_{\Omega}(x), \ x\in M-\partial\Omega,\,
\end{equation}
where $\textbf{1}_{\Omega}:M\mapsto\{0,1\}$ is the characteristic
function of $\Omega$, and where the convergence in \eqref{e7} is locally uniform. It can be shown that if $|\Omega|<\infty$, then the convergence is also in $L^1(M)$. If $\Omega$ has infinite measure and  $|\partial \Omega|=0$, then the convergence is also in $L^1_{\textup{loc}}(M)$ (Section 7.4 in \cite{GB1}).

In this paper we obtain bounds for the heat content in the case where $\Omega$ has possibly infinite measure or infinite perimeter, and where $M$ satisfies the following condition.

There exists $C\in [2,\infty)$ such that for all $x\in M,\,y\in M,\,t>0,\, R>0$,
\begin{equation}\label{e8}
\frac{e^{-Cd(x,y)^2/t}}{C\sqrt{|B(x;t^{1/2})||B(y;t^{1/2})|}}\le p_M(x,y;t)\le \frac{Ce^{-d(x,y)^2/(Ct)}}{\sqrt{|B(x;t^{1/2})||B(y;t^{1/2})|}},
\end{equation}
and
\begin{equation}\label{e9}
|B(x;2R)|\le C |B(x;R)|,
\end{equation}
where $B(x;R)=\{y\in M:d(x,y)<R\}$, and $d(x,y)$ denotes the geodesic distance between $x$ and $y$.

It was shown independently in \cite{GB2} and \cite{LSC1} that $M$ satisfying a volume doubling property and a Poincar\'e inequality is equivalent to $M$ satisfying a parabolic Harnack principle, and is also equivalent to the Li-Yau bound \eqref{e8} above. See for example Theorem 5.4.12 in \cite{LSC}. We included \eqref{e9} in the definition of the constant $C$, even though the volume doubling property is implied by \eqref{e8}.

We recall a few basic facts.\\
\noindent\textup{(i)} Volume doubling implies that for $x\in M,r_0>0$,
\begin{equation*}%\label{e10}
\int_{r_0}^{\infty}dr\, r(\log \vert B(x;r)\vert)^{-1}=+\infty.
\end{equation*}
Hence $u_{\Omega}$, defined by \eqref{e6}, is the unique, bounded solution of \eqref{e1} with initial condition \eqref{e7} in the sense of $L^1_{\textup{loc}}(M)$.
Moreover stochastic completeness holds. That is for all $x\in M,\,t>0$,
\begin{equation}\label{e11}
\int_Mdy\,p_M(x,y;t)=1.
\end{equation}
 We refer to Chapter 9 in \cite{GB}.\\
\noindent\textup{(ii)} If $H_{\Omega}(t)<\infty$ for all $t>0$, then for all $t>0,\,s>0$ we have by Cauchy-Schwarz's inequality, \eqref{e2} and \eqref{e6} that
\begin{align*}%\label{e12}
H_{\Omega}((t+s)/2)&=\int_Mdz\,\int_{\Omega}dy\,\int_{\Omega}dx\,p_M(x,z;t/2)p_M(x,z;s/2)\nonumber \\
&=\int_Mdx\,u_{\Omega}(x;t/2)u_{\Omega}(x;s/2)\nonumber \\ &
\le\bigg(\int_Mdx\,u^2_{\Omega}(x;t/2)\bigg)^{1/2}\bigg(\int_Mdx\,u^2_{\Omega}(x;s/2)\bigg)^{1/2}\nonumber \\ &
=\bigg(H_{\Omega}(t)H_{\Omega}(s)\bigg)^{1/2}.
\end{align*}
Hence $t\mapsto H_{\Omega}(t)$ is mid-point log-convex, log-convex, convex, and hence continuous on $(0,\infty)$.\\
\noindent\textup{(iii)} If \eqref{e9} holds for all $x\in M\,,R>0$ then
\begin{equation}\label{e13}
\frac{\vert B(x;r_2)\vert}{\vert B(x;r_1)\vert}\le C \bigg(\frac{r_2}{r_1}\bigg)^{(\log C)/\log 2}\,, r_2\ge r_1.
\end{equation}
We refer to (2.2) in \cite{LSC1}.\\
\noindent\textup{(iv)} $t\mapsto H_{\Omega}(t)$ is decreasing: if $H_{\Omega}(t)<\infty$ for some $t>0$, then for $s>0,$
\begin{align}\label{e14}
H_{\Omega}(t+s)&=\int_{M}dx\,u^2_{\Omega}(x;(t+s)/2)\nonumber \\ &=\int_Mdx\,\int_{M}dy_1\,p_M(x,y_1;s/2)u_{\Omega}(y_1;t/2)\nonumber \\ &  \hspace{5mm}\times\int_{M}dy_2\,p_M(x,y_2;s/2)u_{\Omega}(y_2;t/2)\nonumber \\
&=\int_{M}dy_1\,\int_{M}dy_2\,p_M(y_1,y_2;s)u_{\Omega}(y_1;t/2)u_{\Omega}(y_2;t/2)\nonumber \\
&\le\frac12 \int_{M}dy_1\,\int_{M}dy_2\,p_M(y_1,y_2;s)(u^2_{\Omega}(y_1;t/2)+u^2_{\Omega}(y_2;t/2))\nonumber \\
&=\int_{M}dy_1\,\int_{M}dy_2\,p_M(y_1,y_2;s)u^2_{\Omega}(y_1;t/2)\nonumber \\
&\le\int_{M}dy\,u^2_{\Omega}(y;t/2)\nonumber \\ &=H_{\Omega}(t).
\end{align}

We make the following.
\begin{definition}\label{def1}
For $x\in M$, $\Omega\subset M$, and $R>0$,
\begin{equation*}%\label{e15}
\mu_{\Omega}(x;R)=\vert B(x;R)\cap \Omega\vert,
\end{equation*}
\begin{equation*}%\label{e20}
\nu_{\Omega}(x;R)=|B(x;R)-\Omega|.
\end{equation*}
\end{definition}
Our main result is the following.

\begin{theorem}\label{the1} Let $M$ be a complete, smooth, non-compact, $m$-dimensional Riemannian manifold without boundary, and let $\Omega\subset M$ be open. Suppose that \eqref{e8}, \eqref{e9} holds for some $C\in [2,\infty)$. Then
\begin{enumerate}
\item[\textup{(i)}] If $H_{\Omega}(T)<\infty$ for some $T>0$, then
\begin{equation}\label{e16}
\int_{\Omega}dx \frac{{\mu_{\Omega}(x;t^{1/2})}}{ {\vert
B(x;t^{1/2})}\vert}<\infty,
\end{equation}
for all $t>0$.
\item[\textup{(ii)}] If \eqref{e16} holds for some $t=T>0$, then
\begin{equation}\label{e17}
K_1\int_{\Omega}dx \frac{\mu_{\Omega}(x;t^{1/2})}{\vert
B(x;t^{1/2})\vert}\le H_{\Omega}(t)\le K_2\int_{\Omega}dx
\frac{\mu_{\Omega}(x;t^{1/2})}{\vert B(x;t^{1/2})\vert},
\end{equation}
for all $t>0$, where
\begin{align}\label{e18}
K_1&=C^{-2}e^{-C},\nonumber \\
 K_2&=2C^{15/4}\bigg(C\log\big(2C^{(7/2)+((\log C)/(\log 2))}\big)\bigg)^{(3\log C)/(4\log 2)}.
\end{align}
\end{enumerate}
\end{theorem}

If $\Omega$ has finite Lebesgue measure, then we define the heat loss of $\Omega$ in $M$ at $t$ by
\begin{equation}\label{e19}
F_{\Omega}(t)=|\Omega|-H_{\Omega}(t).
\end{equation}
We have that the heat loss $t\mapsto
F_{\Omega}(t)$ of $\Omega$ in $M$ is increasing, concave, subadditive, and continuous. If
$\Omega$ is bounded and $\partial \Omega$ is smooth, then,
by \eqref{e4}, there exists an asymptotic series of which the first few
coefficients are known explicitly. Theorem \ref{the2} below concerns the general
situation $\vert\Omega\vert<\infty,$ and gives bounds in non-classical geometries where e.g. either
$\Omega$ has infinite perimeter, and/or $\partial\Omega$ is not smooth.

\begin{theorem}\label{the2}Let $M$ be a complete, smooth, non-compact, $m$-dimensional Riemannian manifold without boundary, and let $\Omega\subset M$ be open with finite Lebesgue measure. Suppose that \eqref{e8}, \eqref{e9} holds for some $C\in [2,\infty)$. Then
\begin{equation}\label{e21}
L_1\int_{\Omega}dx \, \frac{\nu_{\Omega}(x;t^{1/2})}{\vert
B(x;t^{1/2})\vert}\le F_{\Omega}(t)\le L_2\int_{\Omega}dx \,
\frac{\nu_{\Omega}(x;t^{1/2})}{\vert B(x;t^{1/2})\vert}, \ \ \textup{for all}\, t>0,
\end{equation}
where
\begin{align}\label{e22}
L_1&=C^{-2}e^{-C},\nonumber \\
L_2&=4C^{15/4}\bigg(C\log\big(2C^{7+((\log C)/\log 2)}\big)\bigg)^{1+((3\log C)/(4\log 2))}.
\end{align}
\end{theorem}

This paper is organised as follows. In Section \ref{sec2} we give the proofs of Theorem \ref{the1} and Theorem \ref{the2}. In Section \ref{sec3} we analyse an example of $\Omega$ in $\R^m$ where precise analysis of $H_{\Omega}(t)$ is possible.

\section{Proofs\label{sec2}}

The main idea in the proof of Theorem \ref{the1} is to use the Li-Yau bound \eqref{e8}, and \eqref{e13} to bootstrap $\int_{\{x\in M-\Omega:\inf_{y\in \Omega}d(x,y)\ge ct^{1/2}\}}dx\,u_{\Omega}(x;t)$ in terms of $H_{\Omega}(t)$. This is possible for $c$ sufficiently large (in terms of $C$).
A similar bootstrap argument features in the proof of Theorem \ref{the2}. There, the stochastic completeness of $M$, \eqref{e11}, is also exploited.

\noindent{\it Proof of Theorem \ref{the1}.}
(i) Let $t\ge T>0,$ and suppose that $H_{\Omega}(T)<\infty$. Let $R>0$. By \eqref{e8} and \eqref{e14} we have that
\begin{align}\label{e23}
H_{\Omega}(T)&\ge H_{\Omega}(t)\nonumber \\ &\ge  \int_{\Omega}dx\int_{\Omega\cap B(x;R)}dy\,
p_M(x,y;t)\nonumber \\ &\ge
C^{-1}e^{-CR^2/t}\int_{\Omega}dx\int_{\Omega\cap B(x;R)}dy\,(\vert
B(x;t^{1/2})\vert \vert B(y;t^{1/2})\vert )^{-1/2}.
\end{align}
For $d(x,y)<R$,  $B(y;t^{1/2})\subset B(x;R+t^{1/2})$, so that by
\eqref{e13},
\begin{equation}\label{e24}
\vert B(y;t^{1/2})\vert\le C\left(\frac{R+t^{1/2}}{t^{1/2}}\right)^{(\log C)/\log 2}
\vert B(x;t^{1/2})\vert.
\end{equation}
The choice $R=t^{1/2}$ implies, by \eqref{e23} and \eqref{e24}, that
\begin{equation}\label{e25}
H_{\Omega}(T)\ge H_{\Omega}(t)\ge K_1\int_{\Omega}dx
\frac{\mu_{\Omega}(x;t^{1/2})}{\vert B(x;t^{1/2})\vert},\ t\ge T,
\end{equation}
with $K_1$ given in \eqref{e18}.

Next suppose that $0<t\le T$. By \eqref{e13}, and \eqref{e25} for $t=T$, we have that
\begin{align*}%\label{e27}
\int_{\Omega}dx \frac{\mu_{\Omega}(x;t^{1/2})}{\vert
B(x;t^{1/2})\vert}&\le
C\left(\frac{T}{t}\right)^{(\log C)/\log 4}\int_{\Omega}dx\frac{\mu_{\Omega}(x;T^{1/2})}{\vert
B(x;T^{1/2})\vert}\nonumber \\ & \le \frac{C}{K_1}\left(\frac{T}{t}\right)^{(\log C)/\log 4}H_{\Omega}(T).
\end{align*}
This completes the proof of the assertion in part (i).

(ii) Let $n>0,\ p\in
\Omega, R>0$, and $\Omega_n=\Omega\cap B(p;n)$, and suppose that
\eqref{e16} holds for some $t=T>0$. Then $\vert\Omega_n\vert \le \vert
B(p;n)\vert<\infty$. Reversing
the roles of $x$ and $y$ in \eqref{e24} we have that for $d(x,y)<R$,
\begin{equation}\label{e28}
\vert B(y;t^{1/2})\vert\ge C^{-1}\left(\frac{t^{1/2}}{R+t^{1/2}}\right)^{(\log C)/\log 2}
\vert B(x;t^{1/2})\vert.
\end{equation}
We have that
\begin{align}\label{e29}
\int_{\Omega_n}&dx\,\int_{\Omega_n}dy\, p_M(x,y;t)\nonumber \\
&=\int_{\Omega_n}dx\,\int_{\Omega_n\cap B(x;R)}dy\,
p_M(x,y;t)+\int_{\Omega_n}dx\,\int_{\Omega_n-B(x;R)}dy\, p_M(x,y;t).
\end{align}
Using \eqref{e8} and \eqref{e28}, we see that
\begin{align}\label{e30}
\int_{\Omega_n}&dx\,\int_{\Omega_n\cap B(x;R)}dy\, p_M(x,y;t)\nonumber
\\ &\le C \int_{\Omega_n}dx\vert
B(x;t^{1/2})\vert^{-1}\int_{\Omega_n\cap B(x;R)}dy
 \left(\frac{\vert B(x;t^{1/2})\vert}{\vert B(y;t^{1/2})\vert}\right)^{1/2}e^{-d(x,y)^2/(Ct)}\nonumber \\
  &\le C^{3/2}\left(\frac{R+t^{1/2}}{t^{1/2}}\right)^{(\log C)/\log 4}\int_{\Omega_n}dx\frac{\mu_{\Omega_n}(x;R)}{\vert B(x;t^{1/2})\vert}.
\end{align}
To bound the second term in the right-hand side of \eqref{e29}, we note that
\begin{equation}\label{e31}
d(x,y)^2/(Ct)\ge
R^2/(2Ct)+d(x,y)^2/(2Ct)\, ,y\in \Omega_n-B(x;R).
\end{equation}
Hence,
\begin{align}\label{e32}
&\int_{\Omega_n}dx\,\int_{\Omega_n-B(x;R)}dy\, p_M(x,y;t)\nonumber \\ &
\le C\int_{\Omega_n}dx\int_{\Omega_n}dy\, (\vert
B(x;t^{1/2})\vert\vert
B(y;t^{1/2})\vert)^{-1/2}e^{-(d(x,y)^2+R^2)/(2Ct)}\nonumber
\\ & \le C^{(5/2)+((\log C)/\log 2)}\int_{\Omega_n}dx\int_{\Omega_n}dy \big(\lvert
B(x;(2C^2t)^{1/2})\rvert \vert
B(y;(2C^2t)^{1/2})\rvert\big)^{-1/2}\nonumber \\ &\hspace{40mm}\times e^{-(d(x,y)^2+R^2)/(2Ct)}\nonumber
\\ &\le C^{(7/2)+((\log C)/\log 2)}e^{-R^2/(2Ct)}\int_{\Omega_n}dx\int_{\Omega_n}dy\,p_{M}(x,y;2C^2t)\nonumber \\ &
\le C^{(7/2)+((\log C)/\log 2)}e^{-
R^2/(2Ct)}H_{\Omega_n}(2C^2t)\nonumber\\&\le C^{(7/2)+((\log C)/\log 2)} e^{-
R^2/(2Ct)}H_{\Omega_n}(t),
\end{align}
where we have used \eqref{e31}, \eqref{e13}, the lower bound in \eqref{e8}, and \eqref{e14}.
We now choose $R^2$ such that the coefficient of
$H_{\Omega_n}(t)$ in the right-hand side of \eqref{e32} is equal to $\frac12$. That is
\begin{equation}\label{e33}
R_*^2=2Ct\log\big(2C^{(7/2)+((\log C)/\log 2)}\big).
\end{equation}
Rearranging and bootstrapping gives, by \eqref{e29}-\eqref{e33}, and the fact that $t^{1/2}\le R_*$, that
\begin{align}\label{e34}
&H_{\Omega_n}(t)\le 2C^{3/2}\left(\frac{R_*+t^{1/2}}{t^{1/2}}\right)^{(\log C)/\log 4}\int_{\Omega}dx\frac{\mu_{\Omega_n}(x;R_*)}{\vert B(x;t^{1/2})\vert }\nonumber \\ &
\le2C^{5/2}\left(\frac{R_*+t^{1/2}}{t^{1/2}}\right)^{(\log C)/\log 4}\left(\frac{R_*}{t^{1/2}}\right)^{(\log C)/\log 2}\int_{\Omega_n}dx\frac{\mu_{\Omega}(x;R_*)}{\vert B(x;R_*)\vert}\nonumber\\&\le 2C^3\left(\frac{R^2_*}{t}\right)^{(3\log C)/(4\log 2)}\int_{\Omega}dx\frac{\mu_{\Omega_n}(x;R_*)}{\vert B(x;R_*)\vert}\nonumber \\ &
\le 2C^3\bigg(2C\log\big(2C^{(7/2)+((\log C)/\log 2)}\big)\bigg)^{(3\log C)/(4\log 2)}\int_{\Omega}dx\frac{\mu_{\Omega_n}(x;R_*)}{\vert B(x;R_*)\vert}.
\end{align}
We choose $t=t_{T}$ such that $R_*=T$, and take the limit $n\rightarrow \infty$ in the right-hand side of \eqref{e34}. This limit is finite by the hypothesis at the beginning of the proof. We conclude that
\begin{align}\label{e34a}
H_{\Omega_n}(t_{T})&\le 2C^3\bigg(2C\log\big(2C^{(7/2)+((\log C)/\log 2)}\big)\bigg)^{(3\log C)/(4\log 2)}\nonumber \\ &\hspace{45mm}\times\int_{\Omega}dx\frac{\mu_{\Omega}(x;T)}{\vert B(x;T)\vert}
\end{align} 
By monotone convergence,
\begin{align}\label{e34b}
H_{\Omega}(t_{T})&\le 2C^3\bigg(2C\log\big(2C^{(7/2)+((\log C)/\log 2)}\big)\bigg)^{(3\log C)/(4\log 2)}\nonumber\\ &\hspace{45mm}\times\int_{\Omega}dx\frac{\mu_{\Omega}(x;T)}{\vert B(x;T)\vert}.
\end{align}
By (i) we obtain that 
\begin{align}\label{e34b}
H_{\Omega}(t)&\le 2C^3\bigg(2C\log\big(2C^{(7/2)+((\log C)/\log 2)}\big)\bigg)^{(3\log C)/(4\log 2)}\nonumber\\ &\hspace{45mm}\times\int_{\Omega}dx\frac{\mu_{\Omega}(x;R_*)}{\vert B(x;R_*)\vert}
\end{align}
for all $t>0$ with $R_*$ given by \eqref{e33}.
Since $H_{\Omega}(t)$ is decreasing in $t$, and since $R_*\ge t^{1/2}$ we conclude from \eqref{e34} that
\begin{align}\label{e35}
H_{\Omega}(R_*^2)&\le 2C^{15/4}\bigg(C\log\big(2C^{(7/2)+((\log C)/(\log 2))}\big)\bigg)^{(3\log C)/(4\log 2)}\nonumber \\ &\hspace{45mm}\times\int_{\Omega}dx\frac{\mu_{\Omega}(x;R_*)}{\vert B(x;R_*)\vert}.
\end{align}
Rescaling $t$ gives the upper bound in \eqref{e17} with $K_2$ given in \eqref{e18}.
This completes the proof of the assertion in part (ii). \hspace*{\fill }$\square $  \\
\noindent{\it Proof of Theorem \ref{the2}.} To prove the lower
bound in \eqref{e21}, we have
by definition of $F_{\Omega}(t)$ in \eqref{e19}, and by \eqref{e11} that
\begin{align}\label{e38}
F_{\Omega}(t)&= \int_{\Omega}dx \int_{M} dy\
p_M(x,y;t)-\int_{\Omega}dx \int_{\Omega}dy\ p_M(x,y;t)\nonumber \\
&=\int_{\Omega}dx\int_{M-\Omega}dy\ p_M(x,y;t).
\end{align}
Hence by \eqref{e8} we have for $R>0$ that
\begin{align*}%\label{e40}
F_{\Omega}(t)&\ge
\int_{\Omega}dx\int_{ B(x;R)-\Omega}dy\,p_M(x,y;t)\nonumber \\
&\ge C^{-1}e^{-CR^2/t}\int_{\Omega}dx\int_{
B(x;R)-\Omega}dy\,\big(\vert B(x;t^{1/2})\vert\vert B(y;t^{1/2})\vert\big)^{-1/2}.
\end{align*}
Since $B(y;t^{1/2})\subset B(x;R+t^{1/2})$, for $y\in B(x;R)$, we
have by \eqref{e24} that
\begin{equation*}%\label{e41}
F_{\Omega}(t)\ge
C^{-3/2}\left(\frac{t^{1/2}}{R+t^{1/2}}\right)^{(\log C)/\log 4}e^{-CR^2/t}\int_{\Omega}dx
\, \frac{\nu_{\Omega}(x;R)}{\vert B(x;t^{1/2})\vert}.
\end{equation*}
The choice $R=t^{1/2}$ gives the lower bound in \eqref{e21}, with $L_1$ given in \eqref{e22}.

To prove the upper bound in \eqref{e21}, we let $R>0$, and write \eqref{e38} as
\begin{align}\label{e42}
F_{\Omega}(t)=&\int_{\Omega}dx\, \int_{(M-\Omega)\cap
B(x;R)}dy\,p_M(x,y;t)\nonumber
\\ &\ +\int_{\Omega}dx\int_{(M-\Omega)\cap
(M-B(x;R))}dy\,p_M(x,y;t).
\end{align}
By \eqref{e8} and \eqref{e28},
\begin{align}\label{e43}
\int_{\Omega}dx\, &\int_{(M-\Omega)\cap B(x;R)}dy\, p_M(x,y;t)\nonumber
\\ &\le C\int_{\Omega}dx\, \int_{(M-\Omega)\cap B(x;R)}dy\,
\big(\vert B(x;t^{1/2})\vert\vert
B(y;t^{1/2})\vert\big)^{-1/2}\nonumber \\ &\le
C^{3/2}\left(\frac{R+t^{1/2}}{t^{1/2}}\right)^{(\log C)/\log 4}\int_{\Omega}dx
\frac{\nu_{\Omega}(x;R)}{\vert B(x;t^{1/2})\vert}.
\end{align}
Furthermore, by \eqref{e8},
\begin{align}\label{e44}
&\int_{\Omega}dx\, \int_{(M-\Omega)\cap
(M-B(x;R))}dy\, p_M(x,y;t)\nonumber \\ & \le
C\int_{\Omega}dx\, \int_{M-\Omega}dy\frac{e^{-(d(x,y)^2+R^2)/(2Ct)}}{(\vert
B(x;t^{1/2})\vert\vert
B(y;t^{1/2})\vert)^{1/2}}\nonumber \\
&\le
C^{(5/2)+((\log C)/\log 2)}\int_{\Omega}dx \int_{M-\Omega}dy \frac{e^{-(d(x,y)^2+R^2)/(2Ct)}}
{(\vert B(x;(2C^2t)^{1/2})\vert\vert
B(y;(2C^2t)^{1/2})\vert)^{1/2}}\nonumber \\ & \le C^{(7/2)+((\log C)/\log 2)}
e^{-R^2/(2Ct)}\int_{\Omega}dx\int_{M-\Omega}dy\, p_M(x,y;2C^2t)\nonumber
\\ & \le
C^{(7/2)+((\log C)/\log 2)}e^{-R^2/(2Ct)}F_{\Omega}(2C^2t).
\end{align}
Since $F$ is subadditive with $F(0)=0$ and $C\ge 2$, we have that $$F_{\Omega}(2C^2t)\le F_{\Omega}(([2C^2]+1)t)\le([2C^2]+1)F_{\Omega}(t)\le C^{7/2}F_{\Omega}(t).$$ Hence, by \eqref{e44},
\begin{align}\label{e45}
\int_{\Omega}dx\, &\int_{(M-\Omega)\cap(M-B(x;R))}dy\, p_M(x,y;t)\le C^{7+((\log C)/\log 2)}e^{-R^2/(2Ct)}F_{\Omega}(t).
\end{align}
We choose $R^2$ such that the coefficient of $F_{\Omega}(t)$ in \eqref{e45} is
equal to $\frac12$. That is
 \begin{equation}\label{e46}
R_*^2=2Ct\log\big(2C^{7+((\log C)/\log 2)}\big).
\end{equation}
Rearranging and bootstrapping gives, by \eqref{e42}-\eqref{e45}, that
\begin{align}\label{e47}
F_{\Omega}(t)&\le 2C^{3/2}\left(\frac{R_*+t^{1/2}}{t^{1/2}}\right)^{(\log C)/\log 4}\int_{\Omega}dx
\frac{\nu_{\Omega}(x;R_*)}{\vert B(x;t^{1/2})\vert}\nonumber \\ &\le
2C^{5/2}\left(\frac{R_*+t^{1/2}}{t^{1/2}}\right)^{(\log C)/\log 4}\left(\frac{R_*}{t^{1/2}}\right)^{(\log C)/\log 2}\int_{\Omega}dx
\frac{\nu_{\Omega}(x;R_*)}{\vert B(x;R_*)\vert}\nonumber \\ & \le 2C^3\left(\frac{R^2_*}{t}\right)^{(3\log C)/(4\log 2)}\int_{\Omega}dx
\frac{\nu_{\Omega}(x;R_*)}{\vert B(x;R_*)\vert}.
\end{align}
Since $t\mapsto F_{\Omega}(t)$ is concave, with $F_{\Omega}(t)\ge0,$ we see that
\begin{equation}\label{e48}
F_{\Omega}(t)\ge \frac{t}{R_*^2}F_{\Omega}(R_*^2).
\end{equation}
Combining \eqref{e46}-\eqref{e48}, gives that
\begin{align*}%\label{e49}
F_{\Omega}(R^2_*)&\le 4C^{15/4}\bigg(C\log\big(2C^{7+((\log C)/\log 2)}\big)\bigg)^{1+((3\log C)/(4\log 2))}\nonumber \\ & \hspace{45mm}\times\int_{\Omega}dx
\frac{\nu_{\Omega}(x;R_*)}{\vert B(x;R_*)\vert}.
\end{align*}
This gives, after rescaling $t$, the upper bound in \eqref{e21} with $L_2$ given in \eqref{e22}.
This completes the proof of Theorem \ref{the2}.\hspace*{\fill }$\square $

\section{Analysis of an example\label{sec3}}

In this section we present the asymptotic analysis of $H_{\Omega}(t)$ as $t\downarrow 0$, of an open set $\Omega$ in $M=\R^m$ consisting of disjoint balls with centres in $\Z^m$, and decreasing radii. Recall that $p_{\R^m}(x,y;t)=(4\pi t)^{-m/2}e^{-|x-y|^2/(4t)}.$ Let
\begin{equation}\label{e51}
\Omega=\cup_{i\in \N} B(z_i;r_i),
\end{equation}
where $(z_i)_{i\in \N}$ is an enumeration of $\Z^m$, and where
$r_1\ge r_2\ge \dots$.
Furthermore, let
\begin{equation}\label{e52}
\delta=1-2r_1>0.
\end{equation}
Theorem \ref{the3} (ii) below asserts that if $H_{\Omega}(t)<\infty$ for all $t>0$, and if \eqref{e52} holds then the balls loose heat independently as $t\downarrow 0$ up to a term exponentially small  in $t$.

\begin{theorem}\label{the3}
\textup{(i)} If $\delta>0$, then $H_{\Omega}(t)<\infty$ for all $t>0$ if and only if
\begin{equation}\label{e53}
\sum_{i=1}^{\infty}r_i^{2m}<\infty.
\end{equation}
\textup{(ii)} If $\delta>0$ and \eqref{e53} holds, then
\begin{equation}\label{e54}
\bigg\lvert H_{\Omega}(t)-\sum_{i=1}^{\infty}H_{B(z_i;r_i)}(t)\bigg\rvert \le \omega_m^2e^{-\delta^2/(8t)}\left(\frac{\sqrt2}{\delta}+\frac{1}{(4\pi t)^{1/2}}\right)^m\sum_{i=1}^{\infty}r_i^{2m},
\end{equation}
where
$\omega_m=\vert B(0;1)\vert.$
\end{theorem}

Below we consider four main regimes: $\frac{1}{2m}<\alpha< \frac{1}{m}$, $\frac{1}{m}<\alpha< \frac{1}{m-1}$, $\frac{1}{m-1}<\alpha< \frac{1}{m-2}$, and $\frac{1}{m-2}<\alpha.$ The latter regime is absent for $m=2$.
In the first regime $\Omega$ has infinite measure, and Theorem \ref{the1} (iii) gives the order of magnitude as $t\downarrow 0$. This has been refined in \eqref{e56}-\eqref{e57} below. In the second regime $\Omega$ has infinite perimeter, and Theorem \ref{the2} gives the order of magnitude as $t\downarrow 0$. This has been refined in \eqref{e63}-\eqref{e64} below. In the third and fourth regimes $\Omega$ has finite perimeter. Theorem \ref{the2} gives  two-sided bounds of order $t^{1/2}$. In \eqref{e60} and \eqref{e62} below we show that the perimeter term appears with the usual numerical constant. The remainder estimates depend on whether $\sum_{i\in \N}r_i^{m-2}$ is infinite or finite. Furthermore there are several borderline cases: $\alpha=\frac{1}{m}, \frac{1}{m-1}, \frac{1}{m-2}$. They all involve logarithmic corrections in the heat content. We only analyse, as an example, the case $\alpha=\frac{1}{m-2}$. The latter case is again absent for $m=2$.
\begin{theorem}\label{the4}
Let $0<a\le\frac14, m\ge 2,$ and let $r_i=ai^{-\alpha},\ i\in \N.$

If $\frac{1}{2m}<\alpha< \frac{1}{m}$ then
\begin{equation}\label{e56}
H_{\Omega}(t)=c_{\alpha,m}t^{(m\alpha-1)/(2\alpha)}+O(1),\ t\downarrow 0,
\end{equation}
where
\begin{align}\label{e57}
c_{\alpha,m}&=2^{m-1-\frac{1}{\alpha}}\pi^{-m/2}\alpha^{-1}\Gamma((2m\alpha-1)/(2\alpha))a^{1/\alpha}\nonumber \\ &\ \ \  \times\int_{B(0;1)}dx \int_{B(0;1)}dy\,\vert x-y\vert^{(1-2m\alpha)/\alpha}.
\end{align}
If $\frac{1}{m}<\alpha< \frac{1}{m-1}$ then
\begin{equation}\label{e58}
F_{\Omega}(t)=d_{\alpha,m}t^{(m\alpha-1)/(2\alpha)} + O(t^{1/2}),\ t\downarrow 0,
\end{equation}
where
\begin{align}\label{e59}
d_{\alpha,m}&=2^{m-1-\frac{1}{\alpha}}\pi^{-m/2}\alpha^{-1}\Gamma((2m\alpha-1)/(2\alpha))a^{1/\alpha}\nonumber \\ &\ \ \  \times\int_{B(0;1)}dx \int_{\R^m-B(0;1)}dy\,\vert x-y\vert^{(1-2m\alpha)/\alpha}.
\end{align}
If $m>2$ and $\frac{1}{m-1}<\alpha<\frac{1}{m-2}$ or if $m=2$ and $\frac{1}{m-1}<\alpha$, then
\begin{equation}\label{e60}
F_{\Omega}(t)=\pi^{-1/2}\textup{Per}(\Omega)t^{1/2}+O(t^{(m\alpha-1)/(2\alpha)}), \ t\downarrow 0.
\end{equation}
If $m>2$ and $\alpha=\frac{1}{m-2}$ then
\begin{equation}\label{e61}
F_{\Omega}(t)=\pi^{-1/2}\textup{Per}(\Omega)t^{1/2}+O\left(t\log \frac{1}{t}\right), \ t\downarrow 0.
\end{equation}
If $m>2$ and $\frac{1}{m-2}<\alpha$ then
\begin{equation}\label{e62}
F_{\Omega}(t)=\pi^{-1/2}\textup{Per}(\Omega)t^{1/2}+O(t), \ t\downarrow 0.
\end{equation}
\end{theorem}

\noindent{\it Proof of Theorem \ref{the3}.} To prove part (i) we first suppose that $H_{\Omega}(t)<\infty$ for some $t>0$. Then
\begin{align}\label{e63}
H_{\Omega}(t)&=\sum_{i=1}^{\infty}\sum_{j=1}^{\infty}\int_{B(z_i;r_i)}dx\int_{B(z_j;r_j)}dy\, p_{\R^m}(x,y;t)\nonumber \\ &\ge
\sum_{i=1}^{\infty}\int_{B(z_i;r_i)}dx\int_{B(z_i;r_i)}dy\, p_{\R^m}(x,y;t)
\nonumber \\ &\ge(4\pi t)^{-m/2}e^{-r_1^2/t}\omega_m^2\sum_{i=1}^{\infty}r_i^{2m}.
\end{align}
Next suppose that $\sum_{i\in \N}r_i^{2m}<\infty.$ Then
\begin{align}\label{e64}
\int_{\Omega}dx \frac{\mu_{\Omega}(x;\delta/2)}{\vert
B(x;\delta/2)\vert}&\le \sum_{\{i:r_i\ge \delta/4\}}\int_{B(z_i;r_i)}dx+\sum_{\{i:r_i< \delta/4\}}\int_{B(z_i;r_i)}dx\ \frac{\mu_{\Omega}(x;\delta/2)}{\vert
B(x;\delta/2)\vert}\nonumber \\ &\le \sum_{\{i:r_i\ge \delta/4\}}\int_{B(z_i;r_i)}dx(4r_i/\delta)^m+\omega_m\bigg(\frac{2}{\delta}\bigg)^m\sum_{\{i:r_i< \delta/4\}}r_i^{2m}\nonumber \\ &\le \omega_m\bigg(\frac{4}{\delta}\bigg)^m\sum_{i\in \N}r_i^{2m},
\end{align}
which implies the reverse implication by Theorem \ref{the1} (ii). This proves the assertion under (i).

To prove part (ii) we note that the lower bound in \eqref{e54} follows from the first inequality in \eqref{e63}. To prove the upper bound we observe that if $x\in B(z_i;r_i),y\in B(z_j;r_j),i\ne j$, then $$\vert x- y \vert\ge \vert z_i-z_j\vert -\vert z_i-x\vert-\vert y-z_j\vert
\ge \vert z_i-z_j\vert+\delta-1\ge \delta\vert z_i-z_j\vert.$$ Hence
\begin{align*}%\label{e65}
&\sum_{i=1}^{\infty}\sum_{\{j\in \N: j\ne i\}}\int_{B(z_i;r_i)}dx\int_{B(z_j;r_j)}dy\, p_{\R^m}(x,y;t)\nonumber \\
&\le \omega_m^2(4\pi t)^{-m/2}e^{-\delta^2/(8t)}\sum_{i=1}^{\infty}\sum_{\{j\in \N: j\ne i\}}r_i^mr_j^me^{-\vert z_i-z_j\vert^2\delta^2/(8t)}\nonumber \\ &
\le \frac12\omega_m^2(4\pi t)^{-m/2}e^{-\delta^2/(8t)}\sum_{i=1}^{\infty}\sum_{\{j\in \N: j\ne i\}}\left(r_i^{2m}+r_j^{2m}\right)e^{-\vert z_i-z_j\vert^2\delta^2/(8t)}\nonumber \\
&\le \omega_m^2(4\pi t)^{-m/2}e^{-\delta^2/(8t)}\sum_{i=1}^{\infty}r_i^{2m}\sum_{z\in \Z^m}e^{-\vert z\vert^2\delta^2/(8t)}\nonumber \\
&\le \omega_m^2(4\pi t)^{-m/2}e^{-\delta^2/(8t)}\sum_{i=1}^{\infty}r_i^{2m}\left(1+2\int_0^{\infty}dz\ e^{-\vert z\vert^2\delta^2/(8t)}\right)^m,
\end{align*}
which gives the bound in \eqref{e54}.\hspace*{\fill }$\square $

\noindent{\it Proof of Theorem \ref{the4}.} We first consider the case $\frac{1}{2m}<\alpha< \frac{1}{m}.$ By \eqref{e54}, it suffices to consider the sum in the left-hand side of \eqref{e59}. Since $r\mapsto H_{B(0;r)}(t)$ is increasing, $i\mapsto H_{B(0;ai^{-\alpha})}(t)$ is decreasing.
Hence
\begin{align}\label{e66}
\sum_{i=1}^{\infty}H_{B(z_i;r_i)}(t)&=\sum_{i=1}^{\infty}H_{B(0;ai^{-\alpha})}(t)\nonumber \\ &
\le \int_0^{\infty}di\ H_{B(0;ai^{-\alpha})}(t)\nonumber \\ &=\int_0^{\infty}di(ai^{-\alpha})^{2m}\int_{B(0;1)}dx \int_{B(0;1)}dy\ p_{\R^m}(axi^{-\alpha},ayi^{-\alpha}; t).
\end{align}
A straightforward application of Tonelli's theorem gives the formulae under \eqref{e56} and \eqref{e57}.
To obtain a lower bound for the left-hand side of \eqref{e66}, we use the monotonicity of $i\mapsto H_{B(0;ai^{-\alpha})}(t)$ once more, and obtain that
\begin{align}\label{e67}
\sum_{i=1}^{\infty}H_{B(z_i;r_i)}(t)&\ge \int_1^{\infty}di\,H_{B(0;ai^{-\alpha})}(t)\nonumber \\ &=c_{\alpha,m}t^{(m\alpha-1)/(2\alpha)}-\int_0^1di\,H_{B(0;ai^{-\alpha})}(t).
\end{align}
The last term in the right-hand side of \eqref{e67} is bounded in absolute value by
\begin{align*}%\label{e73}
\int_0^1di\,H_{B(0;ai^{-\alpha})}(t)\le
\int_0^1di\,\vert B(0;ai^{-\alpha})\vert=\omega_ma^m(1-\alpha m)^{-1}.
\end{align*}
This completes the proof of the assertion under \eqref{e56} and \eqref{e57}.

Consider the case $\frac{1}{m}<\alpha< \frac{1}{m-1}$. By \eqref{e54}, and scaling we have that
\begin{align}\label{e68}
F_{\Omega}(t)&=\sum_{i=1}^{\infty}F_{B(0;ai^{-\alpha})}(t)+O(e^{-\delta^2/(16t)})\nonumber \\ &
=\sum_{i=1}^{\infty}(ai^{-\alpha})^{m}F_{B(0;1)}(a^{-2}i^{2\alpha}t)+O(e^{-\delta^2/(16t)}).
\end{align}
In a similar way to the proof of \eqref{e56},\eqref{e57}, we approximate the sum with respect to $i$ by an integral. However, $i\mapsto F_{B(0;1)}(a^{-2}i^{2\alpha}t)$ is increasing, whereas $i\mapsto(ai^{-\alpha})^{m}$ is decreasing.
\begin{lemma}\label{lem2} Let
$f:\R^+\mapsto \R^+$ be increasing, and let $g:\R^+\mapsto \R^+$ be decreasing. If $fg$ is summable, then
\begin{equation}\label{e69}
\big\vert \sum_{i=1}^{\infty}f(i)g(i)-\int_1^{\infty} dx\, f(x)g(x) \big\vert\le \sum_{i=1}^{\infty}f(i+1)\big(g(i)-g(i+1)\big).
\end{equation}
\end{lemma}
\begin{proof}
We have that  $$\int_i^{i+1}dx\,f(x)g(x)\ge f(i)g(i+1)=f(i)g(i)-f(i)\big(g(i)-g(i+1)\big),$$ and so
\begin{align}\label{e77}
\sum_{i=1}^{\infty}f(i)g(i)\le \int_1^{\infty}dx\, f(x)g(x)+\sum_{i=1}^{\infty}f(i)\big(g(i)-g(i+1)\big).
\end{align}
Similarly
$$\int_i^{i+1}dx\,f(x)g(x)\le f(i+1)g(i)=f(i+1)g(i+1)+f(i+1))\big(g(i)-g(i+1)\big),$$
and
\begin{align*}%\label{e79}
\sum_{i=1}^{\infty}f(i+1)g(i+1)\ge \int_1^{\infty}dx\,f(x)g(x)-\sum_{i=1}^{\infty}f(i+1)\big(g(i)-g(i+1)\big).
\end{align*}
So
\begin{align}\label{e80}
\sum_{i=1}^{\infty}f(i)g(i)\ge \int_1^{\infty}dx\,f(x)g(x)-\sum_{i=1}^{\infty}f(i+1)\big(g(i)-g(i+1)\big).
\end{align}
Inequality \eqref{e69} follows from \eqref{e77}, \eqref{e80} and $f(i)\le f(i+1)$.
\end{proof}
Let $f(x)=F_{B(0;1)}(a^{-2}x^{2\alpha}t)$, and $g(x)=a^mx^{-m\alpha}$. Using $g(x)-g(x+1)\le a^m m\alpha x^{-m\alpha-1}$, we obtain that
\begin{align}\label{e70}
0\le \sum_{i=1}^{\infty}f(i+1)\big(g(i)-g(i+1)\big)&\le a^mm\alpha\sum_{i=1}^{\infty}i^{-m\alpha-1}F_{B(0;1)}(a^{-2}(i+1)^{2\alpha}t)\nonumber \\ &\le
a^mm\alpha\sum_{i=1}^{\infty}i^{-m\alpha-1}F_{B(0;1)}(a^{-2}(2i)^{2\alpha}t)\nonumber \\ &\le a^{m-1}m^2\omega_m\alpha2^{\alpha}\pi^{-1/2}\sum_{i=1}^{\infty}i^{-m\alpha+\alpha-1}t^{1/2}\nonumber \\ &=O(t^{1/2}),
\end{align}
where we have used that (Proposition 8 in \cite{P}) $F_{B(0;1)}(t)\le m\omega_m\pi^{-1/2}t^{1/2}$.
This gives that
\begin{equation}\label{e71}
\int_0^1dx\,f(x)g(x)\le m\omega_m\pi^{-1/2}a^{m-1}t^{1/2}\int_0^1dx\ x^{-m\alpha+\alpha}=O(t^{1/2}).
\end{equation}
By \eqref{e68},\eqref{e69},\eqref{e70}, and \eqref{e71} we conclude that
\begin{align*}%\label{e72}
F_{\Omega}(t)&=\int_0^{\infty}dx\,a^mx^{-m\alpha}F_{B(0;1)}(a^{-2}x^{2\alpha}t)+O(t^{1/2})\nonumber \\ &
=d_{\alpha,m}t^{(m\alpha-1)/(2\alpha)} + O(t^{1/2}),
\end{align*}
where $d_{\alpha,m}$ is given by \eqref{e59}. This completes the proof of \eqref{e58}.

Consider the cases $m>2$ and $\frac{1}{m-1}<\alpha< \frac{1}{m-2}$ or $m=2$ and $\frac{1}{m-1}<\alpha$. Then $\Omega$ has finite measure, and finite perimeter. Let $I\in \N$, and apply Theorem 2 from \cite{vdBGi} to a ball of radius $r$:
\begin{equation}\label{e73}
\vert H_{B(0;r)}-|B(0;r)|+\pi^{-1/2}\textup{Per}(B(0;r))t^{1/2}\vert\le c_mr^{m-2}t,\, t>0,
\end{equation}
where $c_m=2^{m+2}m^3\omega_m$. Then
\begin{align}\label{e74}
H_{\Omega}(t)&\ge \sum_{i=1}^I H_{B(0;ai^{-\alpha})}(t)\nonumber \\ &\ge
\vert\Omega\vert-\pi^{-1/2}\textup{Per}(\Omega)t^{1/2}-\sum_{i=I+1}^{\infty}\omega_ma^mi^{-\alpha m}-c_m\sum_{i=1}^{I}(ai^{-\alpha})^{m-2}t.
\end{align}
The third term in the right-hand side of \eqref{e74} is $O(I^{1-\alpha m})$. The fourth term is
$O(I^{1-\alpha(m-2)})t$. The choice $I=\lfloor t^{-1/(2\alpha)}\rfloor$ gives the $O(t^{(m\alpha-1)/(2\alpha)})$ remainder in the lower bound.

To obtain an upper bound, we let $J\in \N$, and note that by \eqref{e54},
\begin{align}\label{e89}
H_{\Omega}(t)&\le \sum_{i=1}^{\infty} H_{B(0;ai^{-\alpha})}(t)+O(e^{-\delta^2/(16t)})\nonumber \\ &
\le \sum_{i=1}^{J} H_{B(0;ai^{-\alpha})}(t)+\sum_{i=J+1}^{\infty}\vert B(0;ai^{-\alpha})\vert+O(e^{-\delta^2/(16t)})\nonumber \\ &
\le  \sum_{i=1}^{J}\left(\vert B(0;ai^{-\alpha})\vert-\pi^{-1/2}\textup{Per}(B(0;ai^{-\alpha}))t^{1/2}+c_m(ai^{-\alpha})^{m-2}t\right)\nonumber \\ & \ \ \ \ \ \ \ \ \ +\sum_{i=J+1}^{\infty}\vert B(0;ai^{-\alpha})\vert+O(e^{-\delta^2/(16t)})\nonumber \\ &\le
\vert\Omega\vert-\pi^{-1/2}\textup{Per}(\Omega)t^{1/2}+\pi^{-1/2}m\omega_m\sum_{i=J+1}^{\infty}(ai^{-\alpha})^{m-1}t^{1/2}\nonumber \\ &\ \ \ \ \ \ \ \ \ \
+c_m\sum_{i=1}^{J}(ai^{-\alpha})^{m-2}t+O(e^{-\delta^2/(16t)}),\, t\downarrow 0.
\end{align}
The third term in the right-hand side of \eqref{e89} is $O(J^{1-\alpha(m-1)})t^{1/2}$. The fourth term in the right-hand side of \eqref{e89} is $O(J^{1-\alpha(m-2)})t$. The choice $J=\lfloor t^{-1/(2\alpha)}\rfloor$ gives a remainder $O(t^{(m\alpha-1)/(2\alpha)})$ for the upper bound, and completes the proof of \eqref{e60}.

Next consider the case $\alpha=\frac{1}{m-2}$. The sum of the third and fourth terms in the right-hand side of \eqref{e74} equals, up to constants,
$I^{-2/(m-2)}+t\log I$. We now choose $I=\lfloor t^{-(m-2)/2}\rfloor$, and obtain the remainder in \eqref{e61}. Similarly, the sum of the third and fourth terms in the right-hand side of \eqref{e89} is of order $J^{-1/(m-2)}t^{1/2}+t \log J$. We now choose $J=\lfloor t^{-(m-2)/2}\rfloor$ to obtain the same remainder.

Finally, consider the case $m>2,\alpha>(m-2)^{-1}$. Then the uniform remainder in the right-hand side of \eqref{e73} is summable. Hence by \eqref{e54},
\begin{equation*}%\label{e90}
\vert H_{\Omega}(t)-\vert \Omega\vert +\pi^{-1/2}\textup{Per}(\Omega)t^{1/2}\vert\le c_m \sum_{i\in \N}(ai^{-\alpha})^{m-2}t + O(e^{-\delta^2/(16t)}),\, t\downarrow 0,
\end{equation*}
and we obtain \eqref{e62}. \hspace*{\fill }$\square $

\subsection*{Acknowledgement}
The author acknowledges support by The Leverhulme Trust
through International Network Grant \emph{Laplacians, Random Walks, Bose Gas,
Quantum Spin Systems}. It is a pleasure to thank Katie Gittins, Alexander Grigor'yan, and the referee for their helpful suggestions.

\end{document}